\documentclass[12pt]{amsart}
\usepackage[usenames,dvipsnames]{xcolor}
\usepackage[french, english]{babel}
\usepackage[T1]{fontenc}
\usepackage[utf8]{inputenc}
\usepackage{amsmath,amssymb,latexsym, mathrsfs}
\usepackage[shortlabels]{enumitem}
\usepackage[unicode]{hyperref}
\usepackage[alphabetic]{amsrefs}
\usepackage{tikz}
\usetikzlibrary{intersections}
\usepackage{microtype}
\usepackage{todonotes}
\usepackage[margin = 1in]{geometry}
\usepackage{thmtools}
\usepackage{thm-restate}
\usetikzlibrary{decorations.markings}

\theoremstyle{definition}
\newtheorem{theorem}{Theorem}
\newtheorem{lemma}[theorem]{Lemma}

\newtheorem{question}[theorem]{Question}
\newtheorem{corollary}[theorem]{Corollary}

\newtheorem{example}[theorem]{Example}

\newtheorem*{acknowledgement}{Acknowledgement}
\newtheorem*{actdimconjecture}{Action Dimension Conjecture}

\newtheorem{definition}[theorem]{Definition}

\newtheorem*{T2}{Theorem~\ref{t:main}}
\newtheorem*{T3}{Theorem~\ref{thm:TC}}

\def\beqa{\begin{eqnarray}}
\def\eeqa{\end{eqnarray}}
\def\beqa{\begin{eqnarray}}
\def\eeqa{\end{eqnarray}}
\def\to{\mapsto}

\newcommand{\gs}{\sigma}

\newcommand{\gG}{\Gamma} 
\newcommand{\gD}{\Delta}

\newcommand{\Cn}{\mathscr C_n}

\DeclareMathOperator{\cd}{cd}

\DeclareMathOperator{\Conf}{Conf}

\DeclareMathOperator{\TC}{TC}

\def\Z{\mathbb{Z}}

\makeatletter
\@namedef{subjclassname@2020}{\textup{2020} Mathematics Subject Classification}
\makeatother

\title[Products of free groups inside graph braid groups]
{Products of free groups inside graph braid groups}

\input xy
\xyoption{all}

\begin{document}
\author{Kasia Jankiewicz} \address{Department of Mathematical Sciences, University of California--Santa Cruz, Santa Cruz, CA 95064}
\email{kasia@ucsc.edu}

\author{Kevin Schreve} \address{Department of Mathematics, Louisiana State University, Baton Rouge, LA~70806}
\email{kschreve@lsu.edu}
\date{\today}
\maketitle
\begin{abstract}
Given a graph $\gG$ and a number $n$, the associated $n^{th}$ graph braid group $B_n(\gG)$ is the fundamental group of the unordered configuration space of $n$ points on $\gG$. \'{S}wi\k{a}tkowski showed that for a given $\gG$ and $n$ large enough, there is a free abelian subgroup of $B_n(\gG)$ of rank equal to the cohomological dimension of $B_n(\gG)$. In this note, we observe that at the cost of possibly adding additional points, we can find a subgroup of the same cohomological dimension which is a direct product of non-abelian free groups, and give some applications. 
\end{abstract}
\section{Introduction}
 Given a topological space $X$, the \emph{unordered configuration space $\Conf_n(X)$} is the space of all $n$-element subsets of $X$. Formally 
 $$\Conf_n(X) = (X^n -\{x\in X^n: x_i = x_j \text{ for some }i\neq j\})/S_n$$ 
 where $S_n$ is the symmetric group on $n$ elements, and its action on $X^n$ is by permuting the coordinates. Without taking the quotient by the action of $S_n$, we obtain the \emph{ordered configuration space $\widetilde\Conf_n(X)$}, i.e.\ the space of ordered $n$-tuples of disjoint points in $X$. Artin braid groups arise as the fundamental groups of the configuration space $\Conf_n(\mathbb R^2)$ on the plane.
 
 When $X = \gG$ is a finite simplicial graph, the fundamental group of $\Conf_n(\gG)$ is the \emph{graph braid group $B_n(\gG)$}. The \emph{pure graph braid group} $PB_n(\gG)$ is the fundamental group of $\widetilde \Conf_n(\gG)$, and it has index $n!$ in $B_n(\gG)$. Graph braid groups were first systematically studied in \cite{abrams, abramsghrist, ghrist, swiatkowski}.
 
The cellular structure of $\gG$ induces a natural structure of a cube complex on $\gG^n$.
 The \emph{unordered combinatorial $n$-point configuration space} $\Cn(\gG)$ is defined as $$\Cn(\gG) = (\gG^n - \Delta)/S_n$$ where here $\Delta$ denotes the \emph{thick} diagonal, i.e.\ the subcomplex of $\gG^n$ consisting of products of edges where at least two edges intersect, and the $S_n$-action freely permutes the coordinates, as before.

It is known that after sufficiently fine subdivision of $\gG$, $\Cn(\gG)$ is a deformation retract of $\Conf_n(\gG)$. We will always assume this has occurred. 
To be precise, if $\gG$ has at least $n$ vertices, Prue and Scrimshaw, verifying a conjecture of Abrams, showed that it suffices to have all distinct vertices of valency $\ne 2$ separated by at least $(n-1)$-edges and each homotopically non-trivial cycle at a vertex of degree $>2$ contain at least $(n+1)$-edges \cite{PrueScrimshaw}.

Note that the dimension of $\Cn(\gG)$ depends on the number of disjoint edges in $\gG$; in particular sufficiently subdividing edges of $\gG$ increases the dimension of $\Cn(\gG)$ up to $n$. 
On the other hand, \'{S}wi\k{a}tkowski showed that up to homotopy one can often consider a smaller cube complex. Let $m(\gG)$ denote the number of essential vertices of $\gG$, where a vertex is \emph{essential} if it has degree $> 2$. 

 \begin{theorem}\cite{swiatkowski}*{Theorem 0.1}
$\Conf_n(\gG)$ is homotopy equivalent to a cube complex of dimension $\min\{m(\gG), n\}$.  Furthermore, $B_n(\gG)$ contains a subgroup which is free abelian of rank $\min(m(\gG), \left \lfloor{n/2}\right \rfloor )$. 
\end{theorem}

In particular, given any fixed graph $\gG$, the cohomological dimension of $B_n(\gG)$ stabilizes at $m(\gG)$ for $n \ge 2m(\gG)$. 
In this note, we observe that at the cost of adding some points on $\gG$, $B_n(\gG)$ contains a $m(\gG)$-fold direct product of nonabelian free groups.

 \begin{theorem}\label{t:main}
Let $\gG$ be a graph with $m(\gG)$ vertices of degree $\ge 3$, and let $m_3(\gG)$ be the number of degree $3$ vertices. Then for $n \ge  2m(\gG)+m_3(\gG)$, the group $B_{n}(\gG)$ contains an $m(\gG)$-fold direct product of free groups $F_2^{m(\gG)}$. 
 \end{theorem}

 We think of this as a graph braid version of the resolution of the Tits conjecture in \cite{crispparis}, and its proposed generalization in \cite{jankiewiczschreve}. Though we did not find this explicit statement in the literature, it follows immediately from known results. 
 We give a proof using a construction of Crisp-Wiest which provides a right-angled Artin group $A$ and a local isometry from $\Cn(\gG)$ to the corresponding Salvetti complex \cite{cw}, as the fact that this subgroup sits nicely inside an ambient right-angled Artin group is crucial for our applications.  
 In particular, in Section \ref{s:applications}, we improve a recent theorem of Knudsen calculating the topological complexity of graph braid groups \cite{k23}*{Theorem 1.1}.

\begin{theorem}\label{thm:TC}
Let $\gG$ be a graph with $m(\gG)>0$ vertices of degree $\ge 3$, and let $m_3(\gG)$ be the number of degree $3$ vertices. Then for all $n \ge 2m(\gG) + m_3(\gG)$ the topological complexity of $B_n(\gG)$ is $2m(\gG)$.
\end{theorem}

In Section~\ref{sec:action dim}, we use Theorem \ref{t:main} to compute the action dimension of $B_n(\gG)$ for $n\geq 2m(\gG) + m_3(\gG)$, and partially compute its $L^2$-cohomology.

\begin{acknowledgement}
The first author was supported by the NSF grant DMS-2203307 and DMS-2238198. The second author was supported by the NSF grant DMS-2203325. We thank Ben Knudsen and Dan Cohen for reading an earlier draft and providing useful comments. 
\end{acknowledgement}

\section{The Crisp-Wiest procedure}

Recall that given a flag simplicial complex $L$, there is an associated \emph{right-angled Artin group} (RAAG) $A_L$, which has generators corresponding to the vertices of $L$, and where two generators commute if and only if the two corresponding vertices are connected by an edge. Of course, we only need the $1$-skeleton of $L$ to define $A_L$, but it will be convenient to have both descriptions. 
There is a locally CAT(0) cube complex with $\pi_1 = A_L$ called the \emph{Salvetti complex} $S(A_L)$, which is a cubical subcomplex of the torus $\prod_{v \in L^{(0)}} (S^1)^v$. Here, $S^1$ is cellulated with one vertex $e_0$, and for each simplex $\gs \in L$, $(S^1)^\gs$ is the subtorus consisting of points $(x_v)_{v \in V}$ so that $x_v = e_0$ for $v \notin \gs$. The Salvetti complex is the union of these standard subtori:  $$S(A_L) = \bigcup_{\gs \in L} (S^1)^\gs.$$

Crisp and Wiest show that for any graph $\gG$, $\Cn(\gG)$ admits a local isometry to the Salvetti complex of some RAAG \cite{cw}. 
We now explain this construction.
Define a graph $\gD$ whose vertices correspond to edges of $\gG$, and two vertices of $\gD$ are connected by an edge if and only if the edges are disjoint in $\gG$. 
Let $A_\gD$ be the associated RAAG, and $S(A_\gD)$ the associated Salvetti complex. 
We note that if $\gG$ is disconnected, then $\Delta$ is a join, hence $A_\gD$ splits in the obvious  way as a direct product. 

There is a natural map $\Cn(\gG) \rightarrow S(A_\gD)$ defined as follows. First choose an orientation for each edge of $\gG$ and $S(A_\Delta)$. 
Now, send all vertices of $\Cn(\gG)$ to the unique vertex in $S(A_\gD)$. 
Each edge $e$ of $\Cn(\gG)$ corresponds to a choice of edge in $\gG$ and $(n-1)$ disjoint vertices, send this oriented edge to the corresponding oriented edge in $S(A_\gD)$. 
It is easy to see that this extends to a combinatorial map $\psi: \Cn(\gG) \rightarrow S(A_\gD)$, as $k$-cubes in $\Cn(\gG)$ correspond to a choice of $k$ disjoint edges in $\gG$ and $n-k$ disjoint vertices, and hence there is a corresponding $k$-torus in $S(A_\gD)$. 
Crisp and Wiest show in \cite{cw}*{Theorem 2} that $\phi$ is a local isometry. 
Therefore, the induced map on $\pi_1$ induces an injection $B_n(\gG) \rightarrow A_\gD$.



 We now review two easy examples that are important to our argument. 
 
  \begin{example}[Star graphs]\label{example:stars}
 Let $T_k$ be the cone on $k \ge 4$ points, and label the edges $a_1, a_2 \dots a_k$. See Figure~\ref{fig: star}. In this case, the $2$-point configuration space $\mathcal{C}_2(T_k)$ is a graph with non-abelian free fundamental group, and is already sufficiently subdivided. The corresponding Crisp-Wiest RAAG, denoted by $A(T_k)$, is then just a free group with generators $a_1 \dots a_k$. 
 We can choose any two leaves of $T_k$ for a basepoint, say $a_1, a_2$. Then for each $i \neq 1,2$, then we can identify $B_2(T_k)$ using the Crisp-Wiest map with the following rank $k-2$ subgroup of $F_k$ 
 $$\langle a_1^{-1}a_ia_2^{-1}a_1a_i^{-1}a_2 : i=3,\dots, k \rangle.$$
 
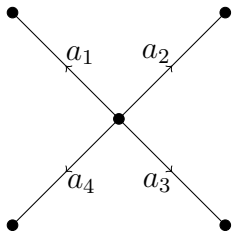
\begin{figure}
\begin{tikzpicture}
\begin{scope}[decoration={
    markings,
    mark=at position 0.5 with {\arrow{>}}}
    ]
\draw[postaction={decorate}]  (0,0) -- (45:2);
\draw[postaction={decorate}]  (0,0) -- (135:2);
\draw[postaction={decorate}]  (0,0) -- (225:2);
\draw[postaction={decorate}]  (0,0) -- (315:2);
\draw[fill=black] (0,0) circle (2pt) ;
\draw[fill=black] (45:2) circle (2pt) ;
\draw[fill=black] (135:2) circle (2pt) ;
\draw[fill=black] (225:2) circle (2pt) ;
\draw[fill=black] (315:2) circle (2pt) ;
\node[left] at (46:1.2) {$a_2$};
\node[right] at (135:1.2) {$a_1$};
\node[right] at (226:1.2) {$a_4$};
\node[left] at (315:1.2) {$a_3$};
\end{scope}
\end{tikzpicture}
\caption{The star $T_4$ on $4$ vertices. The corresponding Crisp-Wiest RAAG for $\mathcal{C}_2(T_4)$ is a free group of rank $4$ freely generated by $a_1, a_2, a_3, a_4$.}
\label{fig: star}
 \end{figure}
 \end{example}
 
  \begin{example}[Tripods]\label{example:tripod}
\begin{figure}
\begin{tikzpicture}

\begin{scope}[decoration={
    markings,
    mark=at position 0.5 with {\arrow{>}}}
    ]

\draw[postaction={decorate}, YellowGreen, thick] (0,0) -- (270:1);
\draw[postaction={decorate}, PineGreen, thick] (270:1) -- (270:2);
\draw[postaction={decorate}, CornflowerBlue, thick] (0,0) -- (30:1);
\draw[postaction={decorate}, RoyalBlue, thick] (30:1) -- (30:2);
\draw[postaction={decorate}, YellowOrange, thick] (0,0) -- (150:1);
\draw[postaction={decorate}, WildStrawberry, thick] (150:1) -- (150:2);

\draw[fill=black] (0,0) circle (1pt) ;
\draw[fill=black] (150:2) circle (1pt) ;
\draw[fill=black] (270:2) circle (1pt) ;
\draw[fill=black] (30:2) circle (1pt) ;
\draw[fill=black] (150:1) circle (1pt) ;
\draw[fill=black] (270:1) circle (1pt) ;
\draw[fill=black] (30:1) circle (1pt) ;
\node[WildStrawberry, above] at (150:1.4) {$a$};
\node[YellowOrange, above] at (150:0.4) {$d$};
\node[RoyalBlue, above] at (30:1.4) {$b$};
\node[CornflowerBlue, above] at (30:0.4) {$e$};
\node[PineGreen, right] at (270:1.5) {$c$};
\node[YellowGreen, right] at (270:0.5) {$f$};
\end{scope}

\begin{scope}[xshift = 6cm, yshift = -2.1cm]
\fill[black!20]  (0:2) -- (180:2) -- (90:3.46) -- (0:2);
\draw (0:2) -- (180:2) -- (90:3.46) -- (0:2);
\draw (60:2) -- (120:2) -- (0,0) -- (60:2);

\draw[WildStrawberry, fill=WildStrawberry] (0,0) circle (2pt) node[below=0.1] { $a$};
\draw[CornflowerBlue, fill=CornflowerBlue] (0:2) circle (2pt) node[below=0.1] { $e$};
\draw[YellowOrange, fill=YellowOrange] (90:3.46) circle (2pt) node[above=0.1] {$d$};
\draw[PineGreen, fill=PineGreen] (60:2) circle (2pt) node[right=0.1] { $c$};
\draw[RoyalBlue, fill=RoyalBlue] (120:2) circle (2pt) node[left=0.1] { $b$};
\draw[YellowGreen, fill=YellowGreen] (180:2) circle (2pt) node[left=0.1] {f};
\end{scope}
\end{tikzpicture}
\caption{The once-subdivided tripod $T_3$ and the corresponding defining flag complex for the Crisp-Wiest RAAG $A(T_3)$ for $\mathcal{C}_3(T_3)$. The use of the colors will be relevant in Example~\ref{ex: non conjugate pair for tripod}.}
\label{fig: tripod}
 \end{figure}
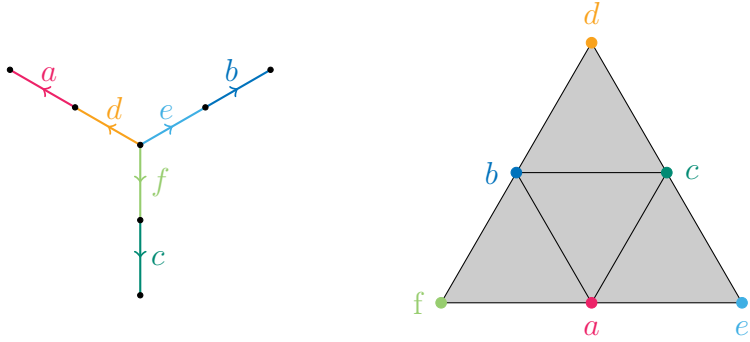
 Let $T_3$ be the cone on $3$ points, and label the edges $a,b$, and $c$. The $2$-point configuration space $\mathcal{C}_2(\gG)$ is a hexagon, and homotopy equivalent to the usual configuration space. If we orient all edges towards the origin, the Crisp-Wiest map induces homomorphism on fundamental groups from $\Z$ to $F_3 = \langle a,b,c \rangle$ which sends the generator of $\Z$ to the word $ac^{-1}ba^{-1}cb^{-1}$.

Now, suppose we have $n = 3$ points. 
We need to subdivide each edge into $2$ intervals, see Figure~\ref{fig: tripod}. The Crisp-Wiest RAAG becomes more complicated, see Figure~\ref{fig: tripod} for its nerve $\Delta$. We shall denote this RAAG by $A(T_3)$. A simple PL Morse theory argument \cite{bestvina} shows that $\mathcal{C}_3(\gG)$ is homotopy equivalent to a wedge of $3$ circles. By choosing the three leaves of $T_3$ as the basepoint of $\Cn(\gG)$, we can identify $B_3(T_3)$ with the subgroup 
$$\langle 
a^{-1}d^{-1}fb^{-1}e^{-1}daf^{-1}eb, 
b^{-1}e^{-1}dc^{-1}f^{-1}ebd^{-1}fc, 
c^{-1}f^{-1}ea^{-1}d^{-1}fce^{-1}da
\rangle\subseteq A_{\Delta}.$$

In general, for $n \ge 3$ the same Morse theoretic argument shows that $B_n(T_3)$ is a non-abelian free group \cite{ghrist, ConnollyDoig}.  
 \end{example}

We now prove Theorem~\ref{t:main}. 

\begin{T2}
Let $\gG$ be a graph with $m(\gG)$ vertices of degree $\ge 3$, and let $m_3(\gG)$ be the number of degree $3$ vertices. Then for $n \ge  2m(\gG)+m_3(\gG)$, the group $B_{n}(\gG)$ contains an $m(\gG)$-fold direct product of free groups $F_2^{m(\gG)}$.
\end{T2}
 
 \begin{proof}[Proof of Theorem \ref{t:main}] Let $n \ge 2m(\gG) + m_3(\gG)$, and take a sufficiently fine edge subdivision of $\gG$ so that $\Cn(\gG)$ is locally CAT(0) cubical with $\pi_1(\Cn(\gG)) = B_n(\gG)$. Let $A_\gD$ be the RAAG given to us by the Crisp-Wiest procedure, so we have an injection $B_n(\gG) \rightarrow A_{\gD}$ induced by the local isometry $\Cn(\gG)\to S(A_\gD)$.

As a basepoint in $\Cn(\gG)$ place $2$ points around each degree $\ge 4$ vertex and place $3$ points two steps away from each degree $3$ vertex, and the rest of the (irrelevant) points somewhere else. This is possible since $n\geq 2m(\gG)+m_3(\gG)$. 

For each essential vertex $v$ of degree $\geq 4$ let $E_v$ be the subset of edges containing it, and for each essential vertex $v$ of degree $3$, let $E_v$ be the subset of edges contained in the $2$-neighborhood of $v$ (i.e.\ $E_v$ forms a once-subdivided tripod).

Let $A(v)$ be the RAAG subgroup of $A_\Delta$ generated by $E_v$, i.e.\ $A(v)\simeq A(T_n)$ where $k$ is the degree of $v$, and where $A(T_k)$ is the Crisp-Wiest RAAG, as in Example~\ref{example:tripod} and Example~\ref{example:stars}. 
Let $A_{\times}$ denote the subgroup of $A_\Delta$ generated by $\bigcup E_v$ where the union is taken over all essential vertices, i.e.\ $A_\times = \prod A(v)$.

For each essential vertex $v$ of degree $3$, the graph braid group $B_n(\gG)$ contains a copy of rank $3$ free group $B_3(T_3)$, which is contained in $A(v)$, as described in Example~\ref{example:tripod}. Similarly, for each essential vertex $v$ of degree $k\geq 4$, $B_n(\gG)$ contains a copy of a rank $k-2$ free group $B_2(T_k)$, contained in $A(v)$, see Example~\ref{example:stars}. We denote such free group by $F(v)$. By passing to a rank $2$ free group of $F(v)$, we obtain $m(\gG)$ copies of rank $2$ free groups in $B_n(\gG)$. Since each such free group belongs to distinct direct factor of $A_\times$, we deduce that there is $m(\gG)$-fold direct product of rank $2$ free groups. 
\end{proof}

\section{Topological complexity}\label{s:applications} Let $P(X)$ be the path space of $X$, i.e.\ the space of all continuous functions $[0,1]\to X$, equipped with the compact-open topology. The path space fibration  $P(X)\rightarrow X\times X$ maps each path $\gamma:[0,1]\to X$ to its endpoints $(\gamma(0), \gamma(1))$. 

The \emph{topological complexity} $\TC(X)$ of a space $X$ is the smallest integer $n$ so that we can cover $X \times X$ by $(n + 1)$ open subsets $U_0, U_1 \dots U_n$ such that the path space fibration $P(X) \rightarrow X \times X$ admits a section over each $U_i$. The topological complexity $\TC(G)$ of a group $G$ is defined to be the topological complexity of its classifying space $BG$, which is well-defined as Farber showed $\TC(X)$ is homotopy invariant. 

Note that $\TC(X) = 0$ if and only if $X$ is contractible. 
It is known that $\TC(G)$ is bounded between $\cd(G)$ and $2\cd(G)$; in general $\TC(X)$ is bounded above by $2\dim X$ \cite{farber}. 
It is typically very difficult to calculate this invariant, even for groups, see however \cite{cv17, dranishnikov, l22} and the references within for some recent results.

Knudsen recently computed the topological complexity of $PB_k$ for $k \ge 2m(\gG)$ and $B_k$ for $k \ge 2m(\gG) + m_3(\gG)$. The computation for $PB_k(\gG)$ verified an old conjecture of Farber.

\begin{theorem}[\cite{k22}, \cite{k23}] Let $\gG$ be a connected graph with $m(\gG) \ge 2$.
\begin{enumerate}

 \item For $n \ge 2m(\gG)$, we have the equality
$$TC(PB_n(\gG)) = 2m(\gG).$$

\item  For $n \ge 2m(\gG) + m_3(\gG)$, we have the equality
$$TC(B_n(\gG)) = 2m(\gG)$$
provided that $\gG$ has no non-separating trivalent vertices. 
\end{enumerate}

\end{theorem}

We shall improve the second part of this theorem by removing the hypothesis of non-separating trivalent vertices, see \cite{k23a}*{Problem 1}.
Our proof strategy is based off of Knudsen's argument in \cite{k23}, but we make use of the Crisp-Wiest procedure to simplify and generalize the argument. 
The following notion is key to our argument (and also to the argument in \cite{k23}). 

\begin{definition}
Let $H_0$ and $H_1$ be subgroups of a group $G$.  We say that $H_0$ and $H_1$ \emph{have disjoint conjugates} if $gH_0g^{-1} \cap H_1 = 1$ for all $g \in G$. 
\end{definition}

\begin{lemma}\label{l:obvious} We record some easy properties.
\begin{enumerate}
\item\label{e:subgp} If $H_0, H_1 < G < G'$ and $H_0$ and $H_1$ have disjoint conjugates in $G'$, then they have disjoint conjugates in $G$. 
\item\label{e:retract} If $H_0$ and $H_1$ have disjoint conjugates in $G$, and $r: G' \rightarrow G$ is a retraction, then $H_0$ and $H_1$ have disjoint conjugates in $G'$. 
\item\label{e:product} If $H_0, H_1$ have disjoint conjugates in $G$ and $H_0', H_1'$ have disjoint conjugates in $G'$, then $H_0 \times H_0', H_1 \times H_1'$ have disjoint conjugates in $G \times G'$. 
\item\label{e:homomorphism} If $f: G \rightarrow K$ is a homomorphism with $H_0 \subset \ker f$ and $f|_{H_1}$ is injective, then $H_0$ and $H_1$ have disjoint conjugates. 
\end{enumerate}
\end{lemma}

\begin{proof}
Part (\ref{e:subgp}) is obvious from the definition.
For (\ref{e:retract}), suppose we have $1 \ne h_0 \in H_0$ and $g \in G'$ with $gh_0g^{-1} \in H_1$. 
Then $$gh_0g^{-1}=r(gh_0g^{-1}) = r(g)h_0r(g)^{-1}.$$ 
Since $r(g)H_0r(g)^{-1}\cap H_1 = 1$, we have $gh_0g^{-1} = 1$, which is a contradiction. 

Part (\ref{e:product}) follows from $g(H_0 \times H_0')g^{-1} = gH_0g^{-1} \times gH_0'g^{-1}$. 
Finally, if $1 \ne h \in gH_0g^{-1}$ then $h \in \ker f$, so $h$ cannot be in $H_1$ as $f|_{H_1}$ is injective, hence (\ref{e:homomorphism}). 
\end{proof}

Our main tool for computing topological complexity is the following theorem of Grant-Lupton-Oprea, which was also one of the main tools of \cite{k22} and \cite{k23}. 

\begin{theorem}\label{dc}\cite{glo}*{Theorem 1.1}
If $H_0, H_1 < G$ have disjoint conjugates, then
$$TC(G) \ge \cd(H_0 \times H_1)$$

\end{theorem}

Note that this gives nontrivial information when $H_0 \times H_1$ is \emph{not} a subgroup of $G$. 

We now state some easy facts about RAAGs that we will use the proof of Theorem~\ref{thm:TC}. 

\begin{lemma}\label{lem:nilpotent}
Let $A_L$ be a RAAG, and $x,y \in A_L$ which generate a rank $2$ free subgroup. Then there is an element $w \in \langle x,y \rangle$ so that the cyclic subgroups $\langle x \rangle$ and $\langle w \rangle$ have disjoint conjugates in $A_L$. 
\end{lemma}

\begin{proof}
Since RAAG's are residually torsion-free nilpotent, there is a torsion-free nilpotent $N$ and a map $f: A_L \rightarrow N$ so that $f(x) \ne 1$. Since $N$ is torsion-free, $f$ is injective on $\langle x \rangle$. Since $\langle x,y \rangle \cong F_2$, the restriction $f_{|\langle x,y\rangle}$ is not injective, so there is an element $w \in \ker f$. Hence, $\langle x \rangle$ and $\langle w \rangle$ have disjoint conjugates by Lemma \ref{l:obvious}(\ref{e:homomorphism}).
\end{proof}

By applying the above lemma to the Crisp-Wiest RAAG $A(T_k)$ of the star graph $T_k$, with the free groups $B_n(T_k)$ where $n = 2$ if $k\geq 4$, and $n=3$ if $k=3$, we deduce that we can find a pair of elements in $B_n(T_k)$ whose cyclic subgroups have disjoint conjugates in $A(T_k)$. In the following two examples we provide explicit examples of such pairs. We will not use the specific elements constructed in those examples, so we do not provide all the details, but we outline how the statements can be verified.

\begin{example} Let $k\geq 4$, and $n=2$. By Example~\ref{example:stars}, the Crisp-Wiest RAAG $A(T_k) \simeq F(a_1,\dots, a_k)$, and $B_2(T_k) \simeq \langle a_1^{-1}a_ia_2^{-1}a_1a_i^{-1}a_2 : i=3,\dots, k \rangle$. Two elements in a free group are conjugate if they have equal cyclic permutations. We note that any two generators of $B_2(T_k)$ are not conjugate in $A(T_k)$, since they do not have equal cyclic permutation. Thus any two of the above generators of $B_2(T_k)$ have disjoint permutations.
\end{example}

\begin{example}\label{ex: non conjugate pair for tripod}
Let now $k=n=3$. By Example~\ref{example:tripod}, the Crisp-Wiest RAAG $A(T_3)$ has the defining graph as illustrated in Figure~\ref{fig: tripod}, and the subgroup $B_3(T_3)$ is a rank $3$ free groups freely generated by 
\begin{align*}
g_1 &:= a^{-1}d^{-1}fb^{-1}e^{-1}daf^{-1}eb
\\
g_2 &:= b^{-1}e^{-1}dc^{-1}f^{-1}ebd^{-1}fc 
\\
g_3 &:= c^{-1}f^{-1}ea^{-1}d^{-1}fce^{-1}da
\end{align*}
Each of $g_1, g_2, g_3$ is conjugate to $e^{-1}fd^{-1}ef^{-1}d$ in $A(T_3)$. See Figure~\ref{fig: cubical completion of a graph} for a geometric explanation.
\begin{figure}
\begin{tikzpicture}
\begin{scope}[decoration={
    markings,
    mark=at position 0.5 with {\arrow{>}}}
    ]
\node[draw, circle, fill = black, inner sep=0pt] (110) at (1,1,0) {};
\node[draw, circle, fill = black, inner sep=0pt] (101) at (1,0,1) {};
\node[draw, circle, fill = black, inner sep=0pt] (011) at (0,1,1) {};
\node[draw, circle, fill = black, inner sep=1pt] (111) at (1,1,1) {};
\node[draw, circle, fill = black, inner sep=0pt] (-111) at (-1,1,1) {};
\node[draw, circle, fill = black, inner sep=0pt] (1-11) at (1,-1,1) {};
\node[draw, circle, fill = black, inner sep=0pt] (11-1) at (1,1,-1) {};

\node[draw, circle, fill = black, inner sep=0pt] (-120) at (-1,2,0) {};
\node[draw, circle, fill = black, inner sep=0pt] (-121) at (-1,2,1) {};
\node[draw, circle, fill = black, inner sep=0pt] (02-1) at (0,2,-1) {};
\node[draw, circle, fill = black, inner sep=0pt] (12-1) at (1,2,-1) {};
\node[draw, circle, fill = black, inner sep=0pt] (-12-1) at (-1,2,-1) {};

\node[draw, circle, fill = black, inner sep=0pt] (20-1) at (2,0,-1) {};
\node[draw, circle, fill = black, inner sep=0pt] (21-1) at (2,1,-1) {};
\node[draw, circle, fill = black, inner sep=0pt] (2-10) at (2,-1,0) {};
\node[draw, circle, fill = black, inner sep=0pt] (2-11) at (2,-1,1) {};
\node[draw, circle, fill = black, inner sep=0pt] (2-1-1) at (2,-1,-1) {};

\node[draw, circle, fill = black, inner sep=0pt] (-102) at (-1,0,2) {};
\node[draw, circle, fill = black, inner sep=0pt] (-112) at (-1,1,2) {};
\node[draw, circle, fill = black, inner sep=0pt] (0-12) at (0,-1,2) {};
\node[draw, circle, fill = black, inner sep=0pt] (1-12) at (1,-1,2) {};
\node[draw, circle, fill = black, inner sep=0pt] (-1-12) at (-1,-1,2) {};

\draw[thick, postaction={decorate}, RoyalBlue] (011) -- (111);
\draw[thick, postaction={decorate}, WildStrawberry] (101) -- (111);
\draw[thick, postaction={decorate}, PineGreen] (110) -- (111);

\draw[thick, postaction={decorate}, CornflowerBlue] (-111) -- (011);

\draw[thick, postaction={decorate}, YellowOrange] (1-11) -- (101);

\draw[thick, postaction={decorate}, YellowGreen] (11-1) -- (110);

\draw[thick, postaction={decorate}, YellowOrange] (-111) -- (-121);
\draw[thick, postaction={decorate}, PineGreen] (-120) -- (-121);

\draw[thick, postaction={decorate}, YellowOrange] (11-1) -- (12-1);
\draw[thick, postaction={decorate}, RoyalBlue] (02-1) -- (12-1);

\draw[thick, postaction={decorate}, CornflowerBlue] (-12-1) -- (02-1);
\draw[thick, postaction={decorate}, YellowGreen] (-12-1) -- (-120);

\draw[thick, postaction={decorate}, CornflowerBlue] (11-1) -- (21-1);
\draw[thick, postaction={decorate}, WildStrawberry] (20-1) -- (21-1);

\draw[thick, postaction={decorate}, CornflowerBlue] (1-11) -- (2-11);
\draw[thick, postaction={decorate}, PineGreen] (2-10) -- (2-11);

\draw[thick, postaction={decorate}, YellowOrange] (2-1-1) -- (20-1);
\draw[thick, postaction={decorate}, YellowGreen] (2-1-1) -- (2-10);

\draw[thick, postaction={decorate}, YellowGreen] (-111) -- (-112);
\draw[thick, postaction={decorate}, WildStrawberry] (-102) -- (-112);

\draw[thick, postaction={decorate}, YellowGreen] (1-11) -- (1-12);
\draw[thick, postaction={decorate}, RoyalBlue] (0-12) -- (1-12);

\draw[thick, postaction={decorate}, CornflowerBlue] (-1-12) -- (0-12);
\draw[thick, postaction={decorate}, YellowOrange] (-1-12) -- (-102);

\draw[->] (0.4, 0.5) to (0.4, -1.2);
\node[left] at (0.4, -0.35) {\scriptsize $g_1$};
\draw[->] (0.5, 0.8) to (-1.2, 0.8);
\node[above] at (-0.35, 0.8) {\scriptsize $g_2$};
\draw[->] (0.9, 0.65) to (1.4, 1.15);
\node[right] at (1.1, 0.8) {\scriptsize $g_3$};

\end{scope}

\begin{scope}[shift = {(5,0)}, decoration={
    markings,
    mark=at position 0.5 with {\arrow{>}}}
    ]
    
\fill[gray!40]  (-1,1,0) -- (-1,1,1) -- (-1,0,1) -- (0,0,1) -- (0,0,0) -- (0,1,0);
\fill[gray!40]  (1,-1,0) -- (1,-1,1) -- (0,-1,1) -- (0,0,1) -- (0,0,0) -- (1,0,0);
\fill[gray!40]  (1,0,-1) -- (1,1,-1) -- (0,1,-1) -- (0,1,0) -- (0,0,0) -- (1,0,0);

\fill[gray!20] (-1,1,0) -- (-1,2,0) -- (-1,2,1) -- (-1,1,1);
\fill[gray!20] (-1,1,1) -- (-1,1,2) -- (-1,0,2,) -- (-1,0,1);

\fill[gray!20] (0,1,-1) -- (0,2,-1) -- (1,2,-1) -- (1,1,-1);
\fill[gray!20] (1,1,-1) -- (2,1,-1) -- (2,0,-1) -- (1,0,-1);

\fill[gray!20] (0,-1,1) -- (0,-1,2) -- (1,-1,2) -- (1,-1,1);
\fill[gray!20] (1,-1,1) -- (2,-1,1) -- (2,-1,0) -- (1,-1,0);

\node[draw, circle, fill = black, inner sep=0pt] (000) at (0,0,0) {};
\node[draw, circle, fill = black, inner sep=0pt] (100) at (1,0,0) {};
\node[draw, circle, fill = black, inner sep=0pt] (010) at (0,1,0) {};
\node[draw, circle, fill = black, inner sep=0pt] (110) at (1,1,0) {};
\node[draw, circle, fill = black, inner sep=0pt] (001) at (0,0,1) {};
\node[draw, circle, fill = black, inner sep=0pt] (101) at (1,0,1) {};
\node[draw, circle, fill = black, inner sep=0pt] (011) at (0,1,1) {};
\node[draw, circle, fill = black, inner sep=1pt] (111) at (1,1,1) {};

\node[draw, circle, fill = black, inner sep=0pt] (-100) at (-1,0,0) {};
\node[draw, circle, fill = black, inner sep=0pt] (-110) at (-1,1,0) {};
\node[draw, circle, fill = black, inner sep=0pt] (-101) at (-1,0,1) {};
\node[draw, circle, fill = black, inner sep=0pt] (-111) at (-1,1,1) {};

\node[draw, circle, fill = black, inner sep=0pt] (0-10) at (0,-1,0) {};
\node[draw, circle, fill = black, inner sep=0pt] (1-10) at (1,-1,0) {};
\node[draw, circle, fill = black, inner sep=0pt] (0-11) at (0,-1,1) {};
\node[draw, circle, fill = black, inner sep=0pt] (1-11) at (1,-1,1) {};

\node[draw, circle, fill = black, inner sep=0pt] (00-1) at (0,0,-1) {};
\node[draw, circle, fill = black, inner sep=0pt] (10-1) at (1,0,-1) {};
\node[draw, circle, fill = black, inner sep=0pt] (01-1) at (0,1,-1) {};
\node[draw, circle, fill = black, inner sep=0pt] (11-1) at (1,1,-1) {};

\node[draw, circle, fill = black, inner sep=0pt] (-120) at (-1,2,0) {};
\node[draw, circle, fill = black, inner sep=0pt] (-121) at (-1,2,1) {};
\node[draw, circle, fill = black, inner sep=0pt] (02-1) at (0,2,-1) {};
\node[draw, circle, fill = black, inner sep=0pt] (12-1) at (1,2,-1) {};
\node[draw, circle, fill = black, inner sep=0pt] (-12-1) at (-1,2,-1) {};

\node[draw, circle, fill = black, inner sep=0pt] (20-1) at (2,0,-1) {};
\node[draw, circle, fill = black, inner sep=0pt] (21-1) at (2,1,-1) {};
\node[draw, circle, fill = black, inner sep=0pt] (2-10) at (2,-1,0) {};
\node[draw, circle, fill = black, inner sep=0pt] (2-11) at (2,-1,1) {};
\node[draw, circle, fill = black, inner sep=0pt] (2-1-1) at (2,-1,-1) {};

\node[draw, circle, fill = black, inner sep=0pt] (-102) at (-1,0,2) {};
\node[draw, circle, fill = black, inner sep=0pt] (-112) at (-1,1,2) {};
\node[draw, circle, fill = black, inner sep=0pt] (0-12) at (0,-1,2) {};
\node[draw, circle, fill = black, inner sep=0pt] (1-12) at (1,-1,2) {};
\node[draw, circle, fill = black, inner sep=0pt] (-1-12) at (-1,-1,2) {};

\draw[thick, postaction={decorate}, RoyalBlue] (000) -- (100);
\draw[thick, postaction={decorate}, RoyalBlue] (010) -- (110);
\draw[thick, postaction={decorate}, RoyalBlue] (001) -- (101);
\draw[thick, postaction={decorate}, RoyalBlue] (011) -- (111);
\draw[thick, postaction={decorate}, WildStrawberry] (000) -- (010);
\draw[thick, postaction={decorate}, WildStrawberry] (100) -- (110);
\draw[thick, postaction={decorate}, WildStrawberry] (001) -- (011);
\draw[thick, postaction={decorate}, WildStrawberry] (101) -- (111);
\draw[thick, postaction={decorate}, PineGreen] (000) -- (001);
\draw[thick, postaction={decorate}, PineGreen] (100) -- (101);
\draw[thick, postaction={decorate}, PineGreen] (010) -- (011);
\draw[thick, postaction={decorate}, PineGreen] (110) -- (111);

\draw[thick, postaction={decorate}, CornflowerBlue] (-100) -- (000);
\draw[thick, postaction={decorate}, CornflowerBlue] (-110) -- (010);
\draw[thick, postaction={decorate}, CornflowerBlue] (-101) -- (001);
\draw[thick, postaction={decorate}, CornflowerBlue] (-111) -- (011);
\draw[thick, postaction={decorate}, PineGreen] (-100) -- (-101);
\draw[thick, postaction={decorate}, PineGreen] (-110) -- (-111);
\draw[thick, postaction={decorate}, WildStrawberry] (-101) -- (-111);
\draw[thick, postaction={decorate}, WildStrawberry] (-100) -- (-110);

\draw[thick, postaction={decorate}, YellowOrange] (0-10) -- (000);
\draw[thick, postaction={decorate}, YellowOrange] (1-10) -- (100);
\draw[thick, postaction={decorate}, YellowOrange] (0-11) -- (001);
\draw[thick, postaction={decorate}, YellowOrange] (1-11) -- (101);
\draw[thick, postaction={decorate}, RoyalBlue] (0-10) -- (1-10);
\draw[thick, postaction={decorate}, RoyalBlue] (0-11) -- (1-11);
\draw[thick, postaction={decorate}, PineGreen] (0-10) -- (0-11);
\draw[thick, postaction={decorate}, PineGreen] (1-10) -- (1-11);

\draw[thick, postaction={decorate}, YellowGreen] (00-1) -- (000);
\draw[thick, postaction={decorate}, YellowGreen] (10-1) -- (100);
\draw[thick, postaction={decorate}, YellowGreen] (01-1) -- (010);
\draw[thick, postaction={decorate}, YellowGreen] (11-1) -- (110);
\draw[thick, postaction={decorate}, RoyalBlue] (00-1) -- (10-1);
\draw[thick, postaction={decorate}, RoyalBlue] (01-1) -- (11-1);
\draw[thick, postaction={decorate}, WildStrawberry] (10-1) -- (11-1);
\draw[thick, postaction={decorate}, WildStrawberry] (00-1) -- (01-1);

\draw[thick, postaction={decorate}, YellowOrange] (-110) -- (-120);
\draw[thick, postaction={decorate}, YellowOrange] (-111) -- (-121);
\draw[thick, postaction={decorate}, PineGreen] (-120) -- (-121);

\draw[thick, postaction={decorate}, YellowOrange] (01-1) -- (02-1);
\draw[thick, postaction={decorate}, YellowOrange] (11-1) -- (12-1);
\draw[thick, postaction={decorate}, RoyalBlue] (02-1) -- (12-1);

\draw[thick, postaction={decorate}, CornflowerBlue] (-12-1) -- (02-1);
\draw[thick, postaction={decorate}, YellowGreen] (-12-1) -- (-120);

\draw[thick, postaction={decorate}, CornflowerBlue] (10-1) -- (20-1);
\draw[thick, postaction={decorate}, CornflowerBlue] (11-1) -- (21-1);
\draw[thick, postaction={decorate}, WildStrawberry] (20-1) -- (21-1);

\draw[thick, postaction={decorate}, CornflowerBlue] (1-10) -- (2-10);
\draw[thick, postaction={decorate}, CornflowerBlue] (1-11) -- (2-11);
\draw[thick, postaction={decorate}, PineGreen] (2-10) -- (2-11);

\draw[thick, postaction={decorate}, YellowOrange] (2-1-1) -- (20-1);
\draw[thick, postaction={decorate}, YellowGreen] (2-1-1) -- (2-10);

\draw[thick, postaction={decorate}, YellowGreen] (-101) -- (-102);
\draw[thick, postaction={decorate}, YellowGreen] (-111) -- (-112);
\draw[thick, postaction={decorate}, WildStrawberry] (-102) -- (-112);

\draw[thick, postaction={decorate}, YellowGreen] (0-11) -- (0-12);
\draw[thick, postaction={decorate}, YellowGreen] (1-11) -- (1-12);
\draw[thick, postaction={decorate}, RoyalBlue] (0-12) -- (1-12);

\draw[thick, postaction={decorate}, CornflowerBlue] (-1-12) -- (0-12);
\draw[thick, postaction={decorate}, YellowOrange] (-1-12) -- (-102);

\node[right] at (0,-2.1) { \scriptsize $X$};
\end{scope}

\begin{scope}[shift = {(10,-1.5)}, decoration={
    markings,
    mark=at position 0.5 with {\arrow{>}}}
    ]
    
\fill[gray!20] (1,1,1) -- (0,1,1) -- (-1,1,1) -- (-1,0,1) -- (0,0,1) -- (1,0,1) -- (1,1,1);
\fill[gray!20] (0,0,1) -- (1,0,1) -- (1,-1,1) -- (0,-1,1);
\fill[gray!20] (-1,1,1) -- (-1,1,2) -- (-1,0,2,) -- (-1,0,1);
\fill[gray!20] (0,-1,1) -- (0,-1,2) -- (1,-1,2) -- (1,-1,1);

\node[draw, circle, fill = black, inner sep=0pt] (001) at (0,0,1) {};
\node[draw, circle, fill = black, inner sep=0pt] (101) at (1,0,1) {};
\node[draw, circle, fill = black, inner sep=0pt] (011) at (0,1,1) {};
\node[draw, circle, fill = black, inner sep=1pt] (111) at (1,1,1) {};

\node[draw, circle, fill = black, inner sep=0pt] (-101) at (-1,0,1) {};
\node[draw, circle, fill = black, inner sep=0pt] (-111) at (-1,1,1) {};

\node[draw, circle, fill = black, inner sep=0pt] (0-11) at (0,-1,1) {};
\node[draw, circle, fill = black, inner sep=0pt] (1-11) at (1,-1,1) {};

\node[draw, circle, fill = black, inner sep=0pt] (-102) at (-1,0,2) {};
\node[draw, circle, fill = black, inner sep=0pt] (-112) at (-1,1,2) {};
\node[draw, circle, fill = black, inner sep=0pt] (0-12) at (0,-1,2) {};
\node[draw, circle, fill = black, inner sep=0pt] (1-12) at (1,-1,2) {};
\node[draw, circle, fill = black, inner sep=0pt] (-1-12) at (-1,-1,2) {};

\draw[thick, postaction={decorate}, RoyalBlue] (001) -- (101);
\draw[thick, postaction={decorate}, RoyalBlue] (011) -- (111);
\draw[thick, postaction={decorate}, WildStrawberry] (001) -- (011);
\draw[thick, postaction={decorate}, WildStrawberry] (101) -- (111);

\draw[thick, postaction={decorate}, CornflowerBlue] (-101) -- (001);
\draw[thick, postaction={decorate}, CornflowerBlue] (-111) -- (011);
\draw[thick, postaction={decorate}, WildStrawberry] (-101) -- (-111);

\draw[thick, postaction={decorate}, YellowOrange] (0-11) -- (001);
\draw[thick, postaction={decorate}, YellowOrange] (1-11) -- (101);
\draw[thick, postaction={decorate}, RoyalBlue] (0-11) -- (1-11);

\draw[thick, postaction={decorate}, YellowGreen] (-101) -- (-102);
\draw[thick, postaction={decorate}, YellowGreen] (-111) -- (-112);
\draw[thick, postaction={decorate}, WildStrawberry] (-102) -- (-112);

\draw[thick, postaction={decorate}, YellowGreen] (0-11) -- (0-12);
\draw[thick, postaction={decorate}, YellowGreen] (1-11) -- (1-12);
\draw[thick, postaction={decorate}, RoyalBlue] (0-12) -- (1-12);

\draw[thick, postaction={decorate}, CornflowerBlue] (-1-12) -- (0-12);
\draw[thick, postaction={decorate}, YellowOrange] (-1-12) -- (-102);

\node[right] at (-2.7,-0.7) { \scriptsize $X_{g_1}$};
\end{scope}

\begin{scope}[shift = {(9,2)}, decoration={
    markings,
    mark=at position 0.5 with {\arrow{>}}}
    ]
    
\fill[gray!40]  (-1,1,0) -- (-1,1,1) -- (-1,0,1) -- (0,0,1) -- (0,0,0) -- (0,1,0);
\fill[gray!40]  (1,-1,0) -- (1,-1,1) -- (0,-1,1) -- (0,0,1) -- (0,0,0) -- (1,0,0);
\fill[gray!40] (1,1,0) -- (0,1,0) -- (0,0,0) -- (1,0,0);
\fill[gray!20] (-1,1,0) -- (-1,2,0) -- (-1,2,1) -- (-1,1,1);
\fill[gray!20] (0,1,-1) -- (0,2,-1) -- (1,2,-1) -- (1,1,-1);
\fill[gray!20] (1,1,-1) -- (2,1,-1) -- (2,0,-1) -- (1,0,-1);
\fill[gray!20] (1,-1,1) -- (2,-1,1) -- (2,-1,0) -- (1,-1,0);

\node[draw, circle, fill = black, inner sep=0pt] (000) at (0,0,0) {};
\node[draw, circle, fill = black, inner sep=0pt] (100) at (1,0,0) {};
\node[draw, circle, fill = black, inner sep=0pt] (010) at (0,1,0) {};
\node[draw, circle, fill = black, inner sep=0pt] (110) at (1,1,0) {};
\node[draw, circle, fill = black, inner sep=0pt] (001) at (0,0,1) {};
\node[draw, circle, fill = black, inner sep=0pt] (101) at (1,0,1) {};
\node[draw, circle, fill = black, inner sep=0pt] (011) at (0,1,1) {};
\node[draw, circle, fill = black, inner sep=1pt] (111) at (1,1,1) {};

\node[draw, circle, fill = black, inner sep=0pt] (-100) at (-1,0,0) {};
\node[draw, circle, fill = black, inner sep=0pt] (-110) at (-1,1,0) {};
\node[draw, circle, fill = black, inner sep=0pt] (-101) at (-1,0,1) {};
\node[draw, circle, fill = black, inner sep=0pt] (-111) at (-1,1,1) {};

\node[draw, circle, fill = black, inner sep=0pt] (0-10) at (0,-1,0) {};
\node[draw, circle, fill = black, inner sep=0pt] (1-10) at (1,-1,0) {};
\node[draw, circle, fill = black, inner sep=0pt] (0-11) at (0,-1,1) {};
\node[draw, circle, fill = black, inner sep=0pt] (1-11) at (1,-1,1) {};

\node[draw, circle, fill = black, inner sep=0pt] (10-1) at (1,0,-1) {};
\node[draw, circle, fill = black, inner sep=0pt] (01-1) at (0,1,-1) {};
\node[draw, circle, fill = black, inner sep=0pt] (11-1) at (1,1,-1) {};

\node[draw, circle, fill = black, inner sep=0pt] (-120) at (-1,2,0) {};
\node[draw, circle, fill = black, inner sep=0pt] (-121) at (-1,2,1) {};
\node[draw, circle, fill = black, inner sep=0pt] (02-1) at (0,2,-1) {};
\node[draw, circle, fill = black, inner sep=0pt] (12-1) at (1,2,-1) {};
\node[draw, circle, fill = black, inner sep=0pt] (-12-1) at (-1,2,-1) {};

\node[draw, circle, fill = black, inner sep=0pt] (20-1) at (2,0,-1) {};
\node[draw, circle, fill = black, inner sep=0pt] (21-1) at (2,1,-1) {};
\node[draw, circle, fill = black, inner sep=0pt] (2-10) at (2,-1,0) {};
\node[draw, circle, fill = black, inner sep=0pt] (2-11) at (2,-1,1) {};
\node[draw, circle, fill = black, inner sep=0pt] (2-1-1) at (2,-1,-1) {};

\draw[thick, postaction={decorate}, RoyalBlue] (000) -- (100);
\draw[thick, postaction={decorate}, RoyalBlue] (010) -- (110);
\draw[thick, postaction={decorate}, RoyalBlue] (001) -- (101);
\draw[thick, postaction={decorate}, RoyalBlue] (011) -- (111);
\draw[thick, postaction={decorate}, WildStrawberry] (000) -- (010);
\draw[thick, postaction={decorate}, WildStrawberry] (100) -- (110);
\draw[thick, postaction={decorate}, WildStrawberry] (001) -- (011);
\draw[thick, postaction={decorate}, WildStrawberry] (101) -- (111);
\draw[thick, postaction={decorate}, PineGreen] (000) -- (001);
\draw[thick, postaction={decorate}, PineGreen] (100) -- (101);
\draw[thick, postaction={decorate}, PineGreen] (010) -- (011);
\draw[thick, postaction={decorate}, PineGreen] (110) -- (111);

\draw[thick, postaction={decorate}, CornflowerBlue] (-100) -- (000);
\draw[thick, postaction={decorate}, CornflowerBlue] (-110) -- (010);
\draw[thick, postaction={decorate}, CornflowerBlue] (-101) -- (001);
\draw[thick, postaction={decorate}, CornflowerBlue] (-111) -- (011);
\draw[thick, postaction={decorate}, PineGreen] (-100) -- (-101);
\draw[thick, postaction={decorate}, PineGreen] (-110) -- (-111);
\draw[thick, postaction={decorate}, WildStrawberry] (-101) -- (-111);
\draw[thick, postaction={decorate}, WildStrawberry] (-100) -- (-110);

\draw[thick, postaction={decorate}, YellowOrange] (0-10) -- (000);
\draw[thick, postaction={decorate}, YellowOrange] (1-10) -- (100);
\draw[thick, postaction={decorate}, YellowOrange] (0-11) -- (001);
\draw[thick, postaction={decorate}, YellowOrange] (1-11) -- (101);
\draw[thick, postaction={decorate}, RoyalBlue] (0-10) -- (1-10);
\draw[thick, postaction={decorate}, RoyalBlue] (0-11) -- (1-11);
\draw[thick, postaction={decorate}, PineGreen] (0-10) -- (0-11);
\draw[thick, postaction={decorate}, PineGreen] (1-10) -- (1-11);

\draw[thick, postaction={decorate}, RoyalBlue] (01-1) -- (11-1);
\draw[thick, postaction={decorate}, WildStrawberry] (10-1) -- (11-1);

\draw[thick, postaction={decorate}, YellowOrange] (-110) -- (-120);
\draw[thick, postaction={decorate}, YellowOrange] (-111) -- (-121);
\draw[thick, postaction={decorate}, PineGreen] (-120) -- (-121);

\draw[thick, postaction={decorate}, YellowOrange] (01-1) -- (02-1);
\draw[thick, postaction={decorate}, YellowOrange] (11-1) -- (12-1);
\draw[thick, postaction={decorate}, RoyalBlue] (02-1) -- (12-1);

\draw[thick, postaction={decorate}, CornflowerBlue] (-12-1) -- (02-1);
\draw[thick, postaction={decorate}, YellowGreen] (-12-1) -- (-120);

\draw[thick, postaction={decorate}, CornflowerBlue] (10-1) -- (20-1);
\draw[thick, postaction={decorate}, CornflowerBlue] (11-1) -- (21-1);
\draw[thick, postaction={decorate}, WildStrawberry] (20-1) -- (21-1);

\draw[thick, postaction={decorate}, CornflowerBlue] (1-10) -- (2-10);
\draw[thick, postaction={decorate}, CornflowerBlue] (1-11) -- (2-11);
\draw[thick, postaction={decorate}, PineGreen] (2-10) -- (2-11);

\draw[thick, postaction={decorate}, YellowOrange] (2-1-1) -- (20-1);
\draw[thick, postaction={decorate}, YellowGreen] (2-1-1) -- (2-10);

\node[right] at (-0.5,-1.75) { \scriptsize $X_{g_2g_3}$};

\end{scope}

\end{tikzpicture}
\caption{The Stallings-folded graph with the fundamental group $B_3(T_3)$ (obtained by folding the rose on three circles labelled by $g_1, g_2, g_3$) can be completed to a cube complex $X$ with the same fundamental group $B_3(T_3)$ that locally isometrically embeds in the Salvetti complex for $A(T_3)$. On the right, the subcomplexes $X_{g_1}, X_{g_2g_3}$ locally isometrically embed in $X$, and have $\pi_1{(X_{g_1})} = \langle g_1\rangle$ and $\pi_{(X_{g_2g_3})} = \langle g_2 g_3\rangle$.}
\label{fig: cubical completion of a graph}
\end{figure}
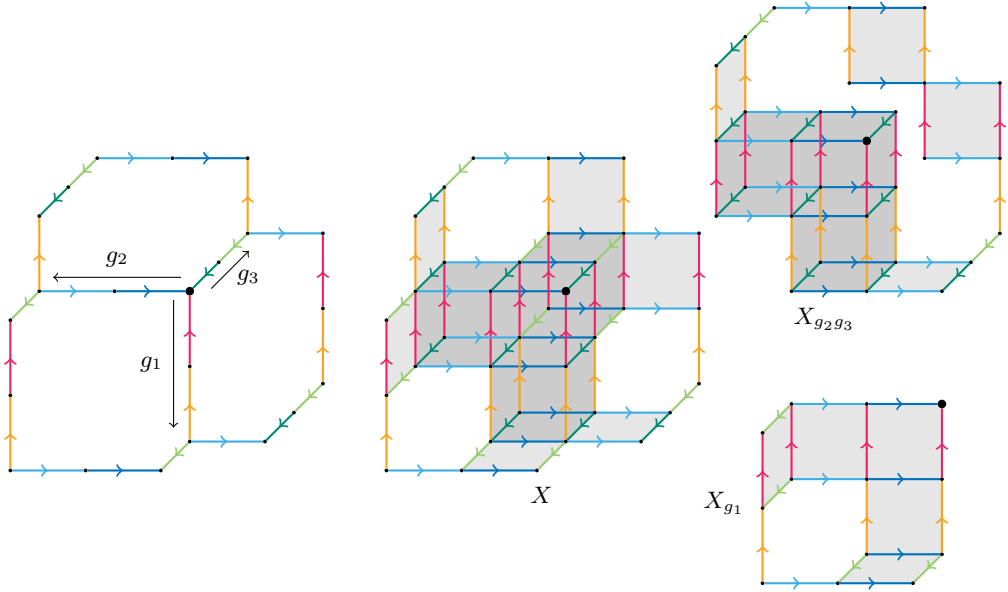
We now sketch an argument why $g_1$ and $g_2g_3$ have disjoint conjugates.
Let $X_{g_1}$ and $X_{g_2g_3}$ be the corresponding cube complexes that locally isometrically embed in the Salvetti complex of $A(T_3)$, illustrated in Figure~\ref{fig: cubical completion of a graph}. The reader can verify that the universal cover of $X_{g_1}$ is \emph{superconvex} in the universal cover of the Salvetti complex, in the sense of \cite[Defn 2.35]{wise}. Then the reader can check directly that the \emph{fiber product} $X_{g_1}\otimes_{X}X_{g_2g_3}$ (see \cite[Defn 8.8]{wise}) does not have any non-contractible components. By \cite[Lem 8.9]{wise} this implies that $\langle g_1\rangle$ and $\langle g_2g_3\rangle$ have disjoint conjugates.\end{example}

\begin{lemma}\label{lem:retraction}
Let $L$ be a simplicial graph, and let $L'$ be a full subgraph of $L$, i.e.\ if two vertices of $L$ are in $L'$, then the edge between them is in $L'$. 
Then $A_L$ retracts onto $A_{L'}$. 
\end{lemma}

\begin{proof}
The map $r:A_L \to A_{L'}$ given by $r(s) = s$ if $s$ is a vertex of $L'$ and $r(s) = 1$ if $s$ is a vertex of $L-L'$ is a retraction. 
\end{proof}

We are now ready to prove Theorem~\ref{thm:TC}. It will follow from Theorem~\ref{dc} where $H_0, H_1$ are appropriately chosen free abelian subgroups of the $m(\gG)$-fold product of free groups obtained in Theorem~\ref{t:main}.

\begin{T3}
Let $\gG$ be a graph with $m(\gG)>0$ vertices of degree $\ge 3$, and let $m_3(\gG)$ be the number of degree $3$ vertices. Then for all $n \ge 2m(\gG) + m_3(\gG)$ the topological complexity of $B_n(\gG)$ is $2m(\gG)$.
\end{T3}

\begin{proof}
Let $n \ge 2m(\gG) + m_3(\gG)$, and take a sufficiently fine subdivision of $\gG$ so that $\Cn(\gG)$ is locally CAT(0) cubical with $\pi_1(\Cn(\gG)) = B_n(\gG)$. 

As in the proof of Theorem \ref{t:main}, we pick the basepoint in $\Cn(\gG)$ is two points around every essential vertex of degree $\ge 4$, three points around every vertex of degree $3$, and the rest of the (irrelevant) points somewhere else.
Also, as in the proof of Theorem \ref{t:main}, let $A_\gD$ be the Crisp-Wiest RAAG with $B_n(\gG) \rightarrow A_{\gD}$, and $A_\times = \prod A(v)$ be a special subgroup, where each $A(v)$ is generated by the the edges $E_v$.
Again, as in the proof of Theorem \ref{t:main}, let $F(v)$ be the free group of $B_n(\gG)$ contained in $A(v)$ and can be naturally identified with $B_2(T_4)$ or $B_3(T_n)$, depending on the degree of the vertex.

By Lemma~\ref{lem:nilpotent} there exist a choice of two cyclic subgroups $C_0(v), C_1(v)$ of $F(v)$ that are not conjugate in $A(v)$. Let $H_i = \langle C_i(v): v \text{ is essential}\rangle$ for $i = 0,1$. Note that each $H_i$ is a free abelian subgroup of $B_n(\gG)$ of rank $m(\gG)$. 
We claim that $H_0$ and $H_1$ have disjoint conjugates in $A_\Delta$. 
By Lemma~\ref{lem:retraction} $A_{\Delta}$ retracts onto $A_\times$. By Lemma \ref{l:obvious} \eqref{e:retract} it suffices to show that $H_0, H_1$ have disjoint conjugates in $A_\times$. 
By assumption, $H_0, H_1$ are generated by cyclic subgroups with disjoint conjugates in each $A(v)$. Since $A_\times$ is a product of $A(v)$ for each essential vertex $v$,  $H_0$ and $H_1$ have disjoint conjugates in $A_\times$ by Lemma \ref{l:obvious}\eqref{e:product}.
By Lemma \ref{l:obvious} \eqref{e:subgp}, this implies that $H_0, H_1$ have disjoint conjugates in $B_n(\gG)$.

By Theorem~\ref{dc}, we conclude that $TC(B_n(\gG))\geq \cd(H_0\times H_1) = 2m(\gG)$.
\end{proof}

\section{Action dimension and $\ell^2$-homology}\label{sec:action dim}

The \emph{action dimension} of a torsion-free group $G$ is the minimal dimension of a manifold model for $BG$, see \cite{bkk} for a more general definition which works for all finitely generated groups. In a similar vein as topological complexity, the action dimension of $G$ is bounded above by twice the geometric dimension of $G$.  

The first nontrivial computation of action dimension was done by Bestvina, Kapovich, and Kleiner; they showed that the action dimension of an $m$-fold product of free groups was as large as possible.
\begin{theorem}[Bestvina-Kapovich-Kleiner \cite{bkk}]\label{t:bkk}
The action dimension of $F_2^m$ is $2m$. 
\end{theorem}

They also showed that the action dimension of the $n$-strand (usual) braid group is $2n+1$, which motivates the following question. 
\begin{question}
Given a graph $\gG$, what is the action dimension of $B_n(\gG)$?
\end{question}

Since action dimension is monotonic with respect to subgroups, we obtain the following corollary of Theorem \ref{t:main} and Theorem \ref{t:bkk}.
\begin{corollary}\label{thm:action dimension}
The action dimension of $B_{n}(\gG)$ is equal to $2m(\gG)$ when $n\geq 2m(\gG) + m_3(\gG)$. 
\end{corollary}

We still think that computing the action dimension for lower $n$ is an interesting problem; in particular it would be nice to know how this relates to the planarity of the graph. 

Action dimension is conjecturally related to the $L^2$-homology groups of $G$.

\begin{actdimconjecture}
Suppose that $G$ is a group with $L^2H_i(G) \ne 0$. Then the action dimension of $G$ is $\ge 2i$. 
\end{actdimconjecture}

It follows from Corollary~\ref{thm:action dimension} that $L^2H_{m(\gG)}(B_n(\gG)) \ne 0$ as long as $n \ge 2m(\gG) + m_3(\gG)$. Indeed, it is known that $L^2H_m(F_2^m) \ne 0$; and if $H < G$, then $L^2H_{\cd(G)}(H) \ne 0$ implies that $L^2H_{\cd(G)}(G) \ne 0$.

\end{document}